\documentclass[english,11pt]{article}

\usepackage[english]{babel}
\usepackage{amsmath,amssymb,enumerate}
\usepackage{amsthm}
\usepackage[latin1]{inputenc}
\usepackage[T1]{fontenc}

\usepackage{psfig,labelfig}
\usepackage{graphicx}
\usepackage[dvips]{epsfig}

\usepackage{color} 

\newtheorem{thm}{Theorem}[section]
\newtheorem{cor}[thm]{Corollary}
\newtheorem{lem}[thm]{Lemma}

\newcommand{\C}{{\mathbb C}}
\newcommand{\Hy}{{\mathbb H}}
\newcommand{\R}{{\mathbb R}}
\newcommand{\Sp}{{\mathbb S}}

\DeclareMathOperator{\arcsinh}{arcsinh}

\DeclareMathOperator{\arctanh}{arctanh}

\DeclareMathOperator{\dist}{dist}
\DeclareMathOperator{\re}{Re}
\DeclareMathOperator{\im}{Im}
\DeclareMathOperator{\up}{up}
\DeclareMathOperator{\rgh}{rgh}

\title{Quasiconformal embeddings of Y-pieces}

\author{Peter Buser, Eran Makover, Bjoern Muetzel and Robert Silhol}

\begin{document}

\maketitle

\begin{abstract}
In this paper we construct quasiconformal embeddings from Y-pieces that contain a short boundary geodesic into degenerate ones. These results are used in a companion paper to study the Jacobian tori of Riemann surfaces that contain small simple closed geodesics.
\end{abstract}

\section{Introduction}

To study the behavior of geometric quantities in degenerating families of compact hyperbolic Riemann surfaces various types of standardized mappings have been devised in the literature that allow one to compare them with the corresponding quantities on the limiting surface. Colbois-Courtois \cite{cc} for instance, use suitably stretching maps of Y-pieces (see a few lines below) to prove convergence of the so-called small eigenvalues of the Laplacian, while Ji \cite{ji} investigates the large eigenfunctions with the help of the infinite energy harmonic maps from \cite{wf}. More recently, for the purpose of constructing quasiconformal deformations of Fuchsian groups with particular limit sets, Bishop \cite{bi1, bi2} makes use of quasiconformal mappings of Y-pieces with exponential behavior in the thin ends that found further applications to quasiconformal mapping class groups, the augmented Teichmüller space and Riemann surfaces of infinite type, \cite{fm},  \cite{mo}, \cite{alp1, alp2}.
 
A \textit{Y-piece} (or \textit{pair of pants}) is a hyperbolic Riemann surface $Y$ of signature $(0,3)$ whose boundary consists of closed geodesics and punctures. If the latter occur $Y$ is called \textit{degenerate}. Y-pieces serve as building blocks: all finite area hyperbolic Riemann surfaces but also many others can be decomposed into or built from them. (See e.g. to \cite{bu}, Chapter 3.)

In the present paper we construct quasiconformal \textit{embeddings} of Y-pieces into degenerate and nearly degenerate ones, together with rather sharp dilatation estimates (\textbf{Theorem~\ref{thm:main}}, \textbf{Theorem~\ref{thm:mainrefined}}). The mappings can be extended in several ways to the Riemann surfaces in which the Y-pieces occur and will be used in \cite{bmms} to study Jacobians of Riemann surfaces that lie close to the boundary of moduli space. By resorting to embeddings rather than surjections we are in a position to get particularly small dilatations. This may perhaps be useful also for other investigations.

From the point of view of quasiconformal mapping theory, as e.g. in Gehring's handbook article \cite{ge}, we are in the sector `triply connected plane domains'. However, the class of mappings studied here is more restrictive in that they are subject to imposed boundary conditions (see \textbf{Theorem~\ref{thm:main}}) that make them applicable to pasting constructions.

\section{Notation and results}\label{sec:results}
We denote by $Y=Y_{l_1,l_2,l_3}$ the Y-piece with boundary geodesics $\gamma_1,\gamma_2,\gamma_3$ of lengths $\ell(\gamma_i)=l_i$, for $i \in \{1,2,3\}.$ The geodesics are parametrized with constant speed $\gamma_i : [0,1] \rightarrow \partial Y$, with positive boundary orientation and such that $\gamma_i(0)$ coincides with the endpoint of the common perpendicular that runs from $\gamma_{i-1}$ to $\gamma_i$ (indices $\mod 3$). We call this the \textit{standard parametrization}. By the \textit{Collar lemma} (see \cite{bu}, p. 106) the sets
\begin{equation}
   \mathcal{C}_i = \{ p \in Y \mid \dist(p,\gamma_i) < \omega_i \} , \text{ \ where \ } \omega_i:= \arcsinh \left(\frac{1}{\sinh(\frac{l_i}{2})}\right)
\label{eq:collars}
\end{equation}
are homeomorphic to $[0,1) \times \Sp^1$ and pairwise disjoint. Moreover, for $0 \leq \rho \leq \omega_i$ the distance sets
\[
     \gamma^\rho_i = \{ p \in Y \mid \dist(p,\gamma_i) = \rho \}
\]
are embedded circles. We parametrize them with constant speed in the form $\gamma^\rho_i:[0,1] \rightarrow Y_{l_1,l_2,l_3}$, such that there is an orthogonal geodesic arc of length $\rho$ from $\gamma_i(t)$ to $\gamma^\rho_i(t), t \in [0,1]$. This too is called a \textit{standard parametrization} and we call the $\gamma^\rho_i$ \textit{equidistant curves}.\\
We extend this to degenerate Y-pieces with cusps writing symbolically `$\ell(\gamma_i)=0$', if the $i$-th end is a puncture. In this case the equidistant curves are horocycles and the collar $\mathcal{C}_i$ is isometric to the surface
\[
    \mathcal{P} = \{ z \in \Hy \mid  \im(z) > \frac{1}{2}\} \mod [z \mapsto z+1],
\]
where $\Hy$ is the hyperbolic plane with the metric
\begin{equation}
    d s^2 = \frac{1}{y^2}(dx^2 + dy^2)
\label{eq:uhp_metric}
\end{equation}
and $[z \mapsto z+1]$ denotes the cyclic group generated by the mapping $m(z) = z+1$.\\
We also consider \textit{restricted Y-pieces}, where parts of the collars have been cut off: Take $Y_{l_1,l_2,l_3}$, select in each $\mathcal{C}_i$ an equidistant curve (respectively, horocycle) $\beta_i$ of length $\lambda_i$, possibly $\beta_i= \gamma_i$, and cut away the outer part of the collar along this curve. This \textit{restricted Y-piece} $Y^{\lambda_1,\lambda_2,\lambda_3}_{l_1,l_2,l_3}$ is the closure of the connected component of $Y_{l_1,l_2,l_3} \backslash \{\beta_1,\beta_2,\beta_3\}$ that has signature $(0,3)$.\\
A homeomorphism
\[
   \phi: Y \rightarrow Y'
\]
of, possibly restricted, Y-pieces is called \textit{boundary coherent} if for corresponding boundary curves $\alpha_i$ of $Y$ and $\alpha'_i$ of $Y'$ in standard parametrization one has $\phi(\alpha_i(t))=\alpha'_i(t),t \in [0,1]$.\\
In the following we denote by $Y^c_{l_1,l_2}$ a surface  $Y^{l_1,l_2,c}_{l_1,l_2,0}$.

\begin{thm} Let $0 \leq l_1,l_2$, $0 < \epsilon \leq \frac{1}{2}$, and set $\epsilon^* = \frac{2}{\pi} \epsilon$. Then there exists a boundary coherent quasiconformal homeomorphism
\[
   \phi : Y_{l_1,l_2,\epsilon} \rightarrow Y^{\epsilon^*}_{l_1,l_2}
\]
with dilatation $q_{\phi} \leq 1+ 2\epsilon^2$.
\label{thm:main}
\end{thm}
This theorem is illustrated in \textit{Figure 1}. Note that the bound is independent of $l_1$ and $l_2$.

Since the Y-pieces in \textbf{Theorem~\ref{thm:main}} are allowed to be degenerate we may apply the theorem twice, so as to get the following.

\begin{figure}[h!]
\SetLabels
\L(.17*.60) $Y_{l_1,l_2,\epsilon}$\\
\L(.75*.60) $Y^{\epsilon^*}_{l_1,l_2}$\\
\L(.49*.65) $\phi$\\
\L(.10*.35) $\,\gamma_1(0)$\\
\L(.33*.80) $\gamma_3(0)$\\
\L(.38*.03) $\gamma_2(0)$\\
\L(.51*.35) $\gamma'_1(0)$\\
\L(.72*.74) $\,h_{\epsilon^*}$\\
\L(.80*.03) $\gamma'_2(0)$\\
\endSetLabels
\AffixLabels{%
\centerline{%
\includegraphics[height=5cm,width=12cm]{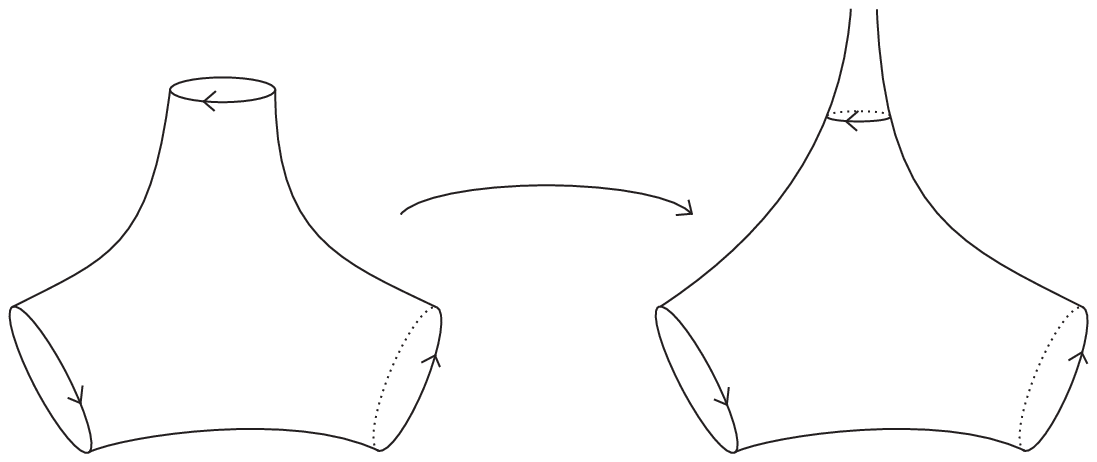}}}
\caption{$Y_{l_1,l_2,\epsilon}$ is quasiconformally embedded into the Y-piece $Y_{l_1,l_2,0}$ with a cusp. The boundary geodesic $\gamma_3$ is sent to the horocycle $h_{\epsilon^*}$ of length $\epsilon^*$.}
\label{fig:theorem1}
\end{figure}

\begin{cor} Let $l_1 \geq 0$ and $0 < \epsilon_2, \epsilon_3 \leq \frac{1}{2}$, and set $\epsilon_i^* = \frac{2}{\pi} \epsilon_i$, $i \in \{2,3\}$. Then there exists a boundary coherent quasiconformal homeomorphism
\[
   \phi : Y_{l_1,\epsilon_2,\epsilon_3} \rightarrow Y^{l_1,\epsilon_2^*, \epsilon_3^*}_{l_1,0,0}
\]
with dilatation $q_{\phi} \leq (1+ 2\epsilon_2^2)(1+ 2\epsilon_3^2) $. 
\label{cor:cor1}
\end{cor}
An enhanced version of \textbf{Theorem~\ref{thm:main}}, the proof of \textbf{Corollary~\ref{cor:cor1}} and some further corollaries will be discussed in Section~\ref{sec:concluding}.\\

The construction of $\phi$ is carried out in Section~\ref{sec:proofthm3}, some technical computations are postponed to Section~\ref{sec:estimates}. 
Throughout the paper we use elementary inequalities for functions in one variable without giving proofs, as e.g. in \eqref{eq:displace_estimate}. They may be checked by plotting the functions out and proved using Taylor polynomials.

\section{Proof of Theorem \ref{thm:main}}\label{sec:proofthm3}
For ease of exposition we restrict ourselves to the case $l_1,l_2 > 0$. The modifications for the remaining cases are straightforward. \\
\\
The common perpendiculars $d_1$ and $d_2$ from $\gamma_3$ to $\gamma_1$ and $\gamma_2$, respectively, together with the common perpendicular $c$ from $\gamma_1$ to $\gamma_2$ decompose $Y_{l_1,l_2,\epsilon}$ into two isometric right-angled geodesic hexagons $\mathcal{H}$ and $\tilde{\mathcal{H}}$. Similarly $Y_{l_1,l_2,0}$ is decomposed into two degenerate hexagons $\mathcal{H}'$ and $\tilde{\mathcal{H}}'$.\\
We shall construct $\phi: \mathcal{H} \rightarrow  \mathcal{H}'$. Extending it symmetrically to $\tilde{\mathcal{H}}$ we then get
\[
\phi: Y_{l_1,l_2,\epsilon} \rightarrow Y^{\epsilon^*}_{l_1,l_2}.
\]
\textit{Figure 2} represents $\mathcal{H}$ and $\mathcal{H}'$ drawn in Fermi coordinates based on the geodesic through $c$, respectively the corresponding perpendicular $c'$ of $\mathcal{H}'$. Thus, sides $c$ and $c'$ lie on the horizontal axis of the coordinate system.
\begin{figure}[h!]
\SetLabels
\L(.09*.02) $C_1$\\
\L(.16*.02) $H_1$\\
\L(.22*.02) $D_1$\\
\L(.26*.02) $D_0$\\
\L(.30*.02) $D_2$\\
\L(.37*.02) $H_2$\\
\L(.44*.02) $C_2$\\
\L(.07*.23) $\alpha_1$\\
\L(.46*.23) $\alpha_2$\\
\L(.07*.41) $A_1$\\
\L(.46*.41) $A_2$\\
\L(.21*.47) $\beta$\\
\L(.13*.53) $\,\,B_1$\\
\L(.39*.53) $B_2$\\
\L(.19*.69) $w$\\
\L(.35*.69) $w$\\
\L(.28*.64) $\lambda$\\
\L(.20*.86) $E_1$\\
\L(.27*.88) $\,\,E_0$\\
\L(.33*.86) $E_2$\\
\L(.57*.02) $C'_1$\\
\L(.64*.02) $H'_1$\\
\L(.71*.02) $D'_0$\\
\L(.78*.02) $H'_2$\\
\L(.85*.02) $C'_2$\\
\L(.55*.23) $\alpha_1$\\
\L(.87*.23) $\alpha_2$\\
\L(.55*.41) $A'_1$\\
\L(.87*.41) $A'_2$\\
\L(.64*.48) $\,\,\, \phi_\beta \circ \beta$\\
\L(.61*.53) $\,B'_1$\\
\L(.80*.53) $B'_2$\\
\endSetLabels
\AffixLabels{%
\centerline{%
\includegraphics[height=7.5cm,width=14cm]{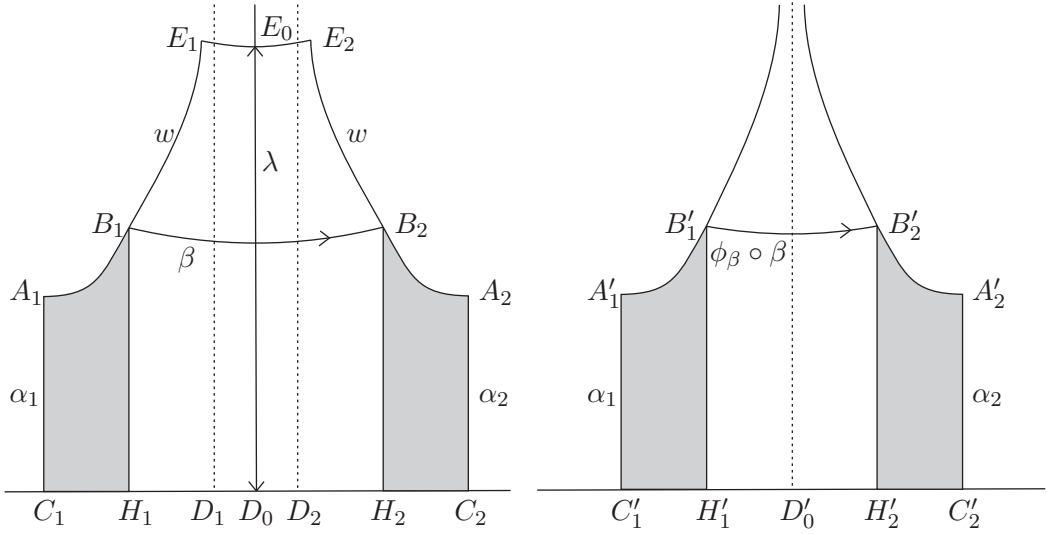}}}
\caption{The geodesic hexagons $\mathcal{H}$ and $\mathcal{H}'$ in Fermi coordinates. $\mathcal{H}'$ is degenerate.}
\label{fig:figure2}
\end{figure}
The sides of $\mathcal{H}$ are in this order $\alpha_1,c,\alpha_2,d_2,\alpha_3,d_1$. With these symbols we denote also the lengths of these curves and we have $\alpha_i= \frac{1}{2} \ell(\gamma_i), i \in \{1,2,3\}$. The vertices are labelled as in \textit{Figure 2} such that
\[
   A_1C_1= \alpha_1, \text{ \ \ } C_1C_2=c, \text{ \ \ } C_2A_2=\alpha_2, \text{ \ \ } A_2E_2=d_2, \text{ \ \ } E_2E_1= \alpha_3 \text{ \ and \ } E_1A_1=d_1.
\]
The points $B_i$ are the endpoints of a (non geodesic) arc $\beta$ that is equidistant to $\alpha_3$ at some distance $w$ that will be introduced in (\ref{eq:define_w}). The points $H_i$ are the endpoints on $c$ of the perpendiculars from $B_i$ to $c$. Finally, $\lambda=E_0D_0$ is the common perpendicular of the sides $\alpha_3$ and $c$. It splits $\mathcal{H}$ into two right-angled geodesic pentagons $\mathcal{P}_i$, $i \in \{1,2\}$, with sides
\begin{equation*}
\alpha_i=A_iC_i, \text{ \ \ } c_i:= C_iD_0, \text{ \ \ } \lambda = D_0E_0, \text{ \ \ } \epsilon_i:= E_0E_i \text{ \ and \ } d_i=E_iA_i.
\label{*eq:pent_points}
\end{equation*}
The vertical, dotted line with endpoint $D_i$ on $c$, $i \in \{1,2\}$, is the vertical geodesic that meets the geodesic line through $A_i$ and $E_i$ at infinity, thus defining a degenerate geodesic pentagon $\mathcal{Q}_i$ with vertices $\infty, A_i,C_i, D_i$ and the two sides of finite length
\[
\alpha_i = A_iC_i \text{ \ and \ } c'_i:= C_iD_i.
\]
For the pentagons the following formulas hold (see \cite{bu}, p. 454).
\begin{eqnarray}
\sinh(\epsilon_i)\sinh(\lambda)&=&\cosh(\alpha_i)   \nonumber \\
\sinh(\alpha_i)\sinh(c_i)&=&\cosh(\epsilon_i) \label{eq:pent_formulas} \\
\sinh(\alpha_i)\sinh(c'_i)&=&1. \nonumber
\end{eqnarray}
The last identity is the same as
\begin{equation}
\sinh(\alpha_i)\cosh(c'_i)=\cosh(\alpha_i).
\label{eq:pent_form2}
\end{equation}
The hexagon $\mathcal{H}'$ on the right-hand side of \textit{Figure 2} has vertices $\infty,A'_1,C'_1,C'_2,A'_2$. The vertical geodesic from $D'_0$ on $c'=C'_1C'_2$ to the point $\infty$ splits $\mathcal{H}'$ into two degenerate pentagons $\mathcal{Q}'_i$. By the last formula in (\ref{eq:pent_formulas}), we have
\[
    C_iD_i =  C'_iD'_0 = c'_i.
\]
Hence, $\mathcal{Q}_i$ and $\mathcal{Q}'_i$ are isometric. \\
In Fermi coordinates horizontal translations are isometries. Hence, $\mathcal{Q}'_i$ is obtained by translating $\mathcal{Q}_i$ horizontally. The displacement length along the horizontal axis for this, in absolute values, is
\begin{equation*}
\delta_i:= D_iD_0.
\label{*eq:displacement}
\end{equation*}
The $\delta_i$ are much smaller than the $\epsilon_i$: using that
\[
   \sinh(c_i) - \sinh(c'_i) =  \sinh(c'_i + \delta_i) - \sinh(c'_i) \geq \delta_i \cosh(c'_i)
\]
we get, using (\ref{eq:pent_formulas}) and (\ref{eq:pent_form2}),
\begin{equation}
\delta_i \leq \frac{\cosh(\epsilon_i)-1}{\cosh(\alpha_i)} \leq \frac{\sinh(\epsilon_i)^2}{2\cosh(\alpha_i)}.
\label{eq:displace_estimate}
\end{equation}
The idea is therefore to map $\mathcal{P}_i$ to $\mathcal{Q}'_i$ by `compressing it horizontally'. However, such a mapping has a small dilatation only up to a certain height. Therefore, we separate $\mathcal{H}$ into two parts $\mathcal{H}_\beta$ and $\mathcal{H}^\beta$ along the arc $\beta$ from $B_1$ to $B_2$ that is equidistant to $\alpha_3$ (see \textit{Figure 2}) at distance $w$, where the following value has proven to be practical:
\begin{equation}
w:=\log \left(\frac{2}{\epsilon}\right).
\label{eq:define_w}
\end{equation}
Since
\[
\sinh(w+ \log(2)) \cdot \sinh(\frac{1}{2} \ell(\gamma_3)) = \frac{1}{2}\left(\frac{4}{\epsilon} - \frac{\epsilon}{4}\right) \sinh(\frac{\epsilon}{2}) < 1
\]
the line $\beta$ lies in the collar $\mathcal{C}_3$ of $\gamma_3$. 

Hence, the endpoints of $\beta$ lie on the sides $d_1$ and $d_2$. Therefore $\beta$ does not intersect side $c$, as drawn in \textit{Figure 2} (i.e. the figure is `correct'). $\mathcal{H}_\beta$ is the part of $\mathcal{H}$ below $\beta$, $\mathcal{H}^\beta$ the part above. \\
We describe the `compression' of $\mathcal{H}_\beta$. First we set
\begin{equation}
\eta_i:= H_iD_0, \text{ \ for \ } i \in \{1,2\}.
\label{eq:etai}
\end{equation}
This $\eta_i$ will play an important role in our construction. Let now $p$ be any point in $\mathcal{H}_\beta$  with Fermi coordinates $(t,r)$, where $t$ is the unit speed parameter of $c$  and $r$ is the distance from $p$ to $c$.  The horizontal positions of $\mathcal{H}$ and $\mathcal{H}'$ are such that $D_0 = c(0) = D'_0$. We define a mapping $F_\beta: \mathcal{H}_\beta \rightarrow \mathcal{H}'_\beta$ given in Fermi coordinates by
\begin{equation}
\begin{gathered}
F_\beta(t,r):= (t',r) \text{ \ \ with \ }  \\
t'=
\left\{ {\begin{array}{*{20}c}
   t+\delta_1  \\
   (1-\frac{\delta_1}{\eta_1}) t  \\
    (1-\frac{\delta_2}{\eta_2}) t  \\
    t-\delta_2  \\
\end{array}} \right.\text{ \ if \ }
\begin{array}{*{20}c}
   t \leq -\eta_1  \\
   -\eta_1 \leq t \leq 0  \\
   0 \leq t \leq \eta_2 \\
   t \geq \eta_2. \\
\end{array}
 \label{eq:Fbeta1}
\end{gathered}
\end{equation}
On the quadrilaterals $B_iA_iC_iH_i$ (shaded areas in \textit{Figure 2}) $F_\beta$ acts isometrically. From the expression of the hyperbolic metric in Fermi coordinates
\begin{equation}
ds^2 = \cosh(r)^2 dt^2 + dr^2
\label{eq:Fermi_metric}
\end{equation}
it follows that $F_\beta$ is quasiconformal with dilatation
\[
   q_{F_\beta} = \frac{1}{1-\frac{\delta_i}{\eta_i}} \text{ \ on \ } \mathcal{H}_{\beta} \cap\mathcal{P}_i , \ i = 1,2.
\]
In fact, $F_\beta$ also has length distortion $\leq q_{F_\beta}$ (see the introduction to \textbf{Lemma~\ref{thm:lemma2}} for the type of length distortion addressed in this paper). In (\ref{eq:upperbound_deltaietai}) we shall see that $\frac{\delta_i}{\eta_i} < \frac{1}{3} \epsilon^2$ so that
\begin{equation}
q_{F_\beta} <  \frac{1}{1-\frac{1}{3} \epsilon^2}.
\label{eq:dilatation_estimate}
\end{equation}
On $\mathcal{H}^\beta$ we shall define a mapping $F^\beta$ that maps equidistant curves of $\alpha_3$ to horocycles. To this end we have to add a modification to $F_\beta$ so that $\beta$ is mapped to a horocycle. Both constructions will be carried out in the upper half plane model of the hyperbolic plane
\[
   \Hy = \{ z = x+iy \in \C \mid y > 0 \}
\]
with the metric (\ref{eq:uhp_metric}), which represents the horocycles nicely. \\
\\
A useful tool for working in $\Hy$ are the isometries $\up(s,\cdot)$ and $\rgh(s,\cdot)$ (`move up' and `move to the right'), defined as Moebius transformations: for $z \in \Hy, s \in \R$,
\begin{equation}
   \up(s,z):= e^s \cdot z, \text{ \ \ \ }  \rgh(s,z):= \frac{\cosh(\frac{s}{2}) \cdot z + \sinh(\frac{s}{2})}{\sinh(\frac{s}{2}) \cdot z + \cosh(\frac{s}{2})}.
\label{eq:define_up_rgh}
\end{equation}
These are hyperbolic isometries with displacement length $|s|$. The axis of $\up(s,\cdot)$ is the positive imaginary axis, the axis of $\rgh(s,\cdot)$ is the geodesic $\Gamma$, passing through $i$ with endpoints at infinity $-1$ and $1$. Its unit speed parametrization is
\begin{equation}
\Gamma(t) = \rgh(t,i) = \tanh(t) + \frac{i}{\cosh(t)}, t \in \R.
\label{eq:define_Gamma}
\end{equation}
The Fermi coordinates $(t,r)$ based on $\Gamma$ of a point $z= x+iy \in \Hy$ are computed as $(t,r) =(\tau(z),\rho(z))$ with the functions
\begin{eqnarray}
\tau(z) &=& \arctanh\left( \frac{2x}{x^2 + y^2 +1}\right) \nonumber \\
\rho(z) &=& \arcsinh \left( \frac{x^2 + y^2 -1}{2y} \right). \label{eq:Fermi_Gamma1}
\end{eqnarray}
Conversely we have
\begin{equation}
z = \frac{\sinh(t)\cosh(r)+i}{\cosh(t)\cosh(r) - \sinh(r)}.
\label{eq:Fermi_Gamma2}
\end{equation}
\begin{proof}
Equation (\ref{eq:Fermi_Gamma2}) is obtained by observing that if $z$ has Fermi coordinates $(t,r)$ then $z= \rgh(t,\up(r,i))$. For (\ref{eq:Fermi_Gamma1}) we plug (\ref{eq:Fermi_Gamma2}) into the right hand side.
\end{proof}
\begin{figure}[h!]
\SetLabels
\L(.11*.02) $-1$\\
\L(.24*.02) $\,\,0$\\
\L(.37*.02) $1$\\
\L(.83*.18) $\,\Gamma$\\
\L(.10*.17) $\,\,\alpha_1$\\
\L(.38*.18) $\alpha_2$\\
\L(.06*.18) $\,A_1$\\
\L(.41*.18) $\,A_2$\\
\L(.145*.135) $\,C_1$\\
\L(.33*.135) $\,C_2$\\
\L(.18*.24) $\,H_1$\\
\L(.30*.24) $H_2$\\
\L(.19*.33) $\eta_1$\\
\L(.29*.33) $\,\eta_2$\\
\L(.26*.34) $i$\\
\L(.17*.54) $\beta$\\
\L(.05*.56) $B_1$\\
\L(.42*.56) $\,B_2$\\
\L(.26*.61) $ie^{\lambda -w}$\\
\L(.08*.70) $\,w$\\
\L(.40*.70) $\,w$\\
\L(.18*.84) $\epsilon_1$\\
\L(.32*.84) $\epsilon_2$\\
\L(.03*.84) $E_1$\\
\L(.26*.90) $E_0=ie^\lambda$\\
\L(.45*.84) $E_2$\\
\L(.12*.90) $\alpha_3$\\
\L(.61*.02) $-1$\\
\L(.74*.02) $\,\,0$\\
\L(.87*.02) $1$\\
\L(.60*.17) $\,\alpha_1$\\
\L(.87*.18) $\alpha_2$\\
\L(.56*.19) $A'_1$\\
\L(.91*.18) $A'_2$\\
\L(.68*.24) $H'_1$\\
\L(.79*.24) $\,H'_2$\\
\L(.75*.34) $\,i$\\
\L(.65*.54) $F_\beta \circ \beta$\\
\L(.56*.56) $B'_1$\\
\L(.91*.56) $\,B'_2$\\
\L(.75*.61) $\,2ia$\\
\endSetLabels
\AffixLabels{%
\centerline{%
\includegraphics[height=7.5cm,width=15cm]{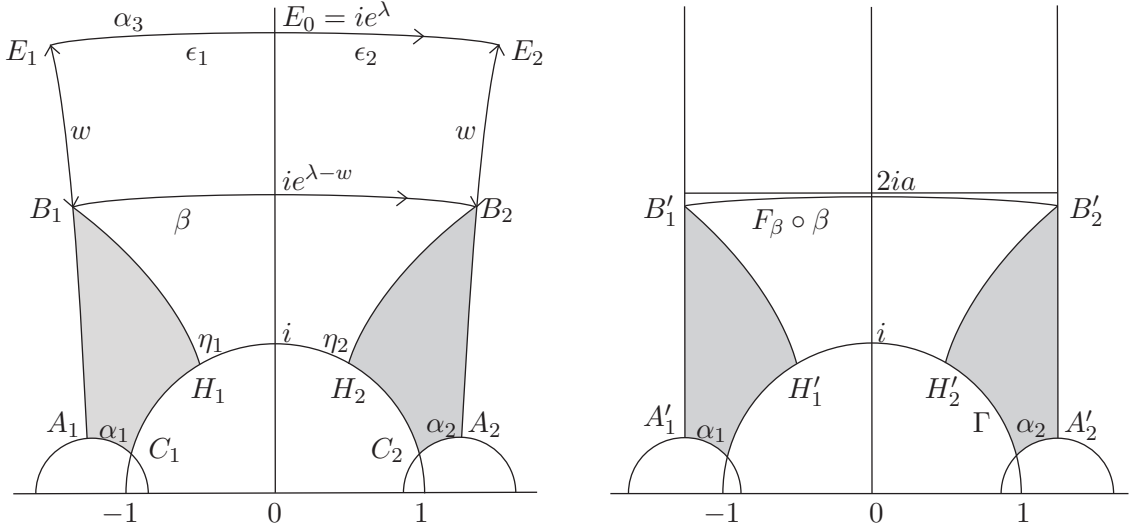}}}
\caption{The geodesic hexagons $\mathcal{H}$ and $\mathcal{H}'$ in the upper half plane $\Hy$. The coordinates of $A_1'$ and $A_2'$ satisfy Equation \eqref{eq:A1A2prime}}.
\label{fig:figure3}
\end{figure}
\textit{Figure 3} shows the hexagons $\mathcal{H}$ and $\mathcal{H}'$ drawn in $\Hy$ with sides $c$ and $c'$ on $\Gamma$ and $D_0=i$, respectively, $D'_0=i$. The geodesic through $D_0\infty$ coincides with the positive imaginary axis. Since it meets the geodesics through $A'_1B'_1$ and through $A'_2B'_2$ at infinity, the latter are vertical straight lines, as drawn in \textit{Figure 3}. Since
\[
  A'_1= \rgh(-c'_1,\up(\alpha_1,i)) \text{ \ and \ } A'_2= \rgh(c'_2,\up(\alpha_2,i))
\]
and as by (\ref{eq:pent_formulas}) $\sinh(\alpha_i)\sinh(c'_i) = 1$, we have
\begin{equation}
 A'_1=-\cosh(\alpha_1) + i \sinh(\alpha_1) \text{ \ and \ }  A'_2=\cosh(\alpha_2) + i \sinh(\alpha_2).
 \label{eq:A1A2prime}
 \end{equation}
 A parametrization of $\beta$ may be obtained as follows: The isometry $m(s,\cdot) = \up(\lambda,\cdot) \circ \rgh(s,\cdot) \circ \up(-\lambda,\cdot), s \in \R$ has oriented displacement length $s$ and its axis runs through $\alpha_3$. Hence, the function $s \mapsto m(s, i e^{\lambda-w})$ is a parametrization of the equidistant curve $\beta$ with speed $\cosh(w)$. Setting
 \begin{equation}
 \kappa = \frac{1}{\cosh(w)}
 \label{eq:define_x}
 \end{equation}
 we get the unit speed parametrization $\beta(s) = m(\kappa s, i e^{\lambda-w})$ i.e.
 \begin{equation}
 \beta(s) = e^\lambda \frac{\sinh(\kappa s)\cosh(w) + i}{\cosh(\kappa s) \cosh(w) + \sinh(w)}, \text{ \ where \ }    -\epsilon_1\cosh(w) \leq s \leq \epsilon_2\cosh(w)
 \label{eq:beta_unit_speed1}
 \end{equation}
 which we may also write in the form
 \begin{equation}
 \beta(s) =  \frac{e^\lambda \sinh(\kappa s)}{\cosh(\kappa s)+\tanh(w)} + i \frac{e^{\lambda-w}}{(\cosh(\kappa s) - 1) \frac{\cosh(w)}{e^w} + 1}.
 \label{eq:beta_unit_speed2}
 \end{equation}
It follows from the definition of $w$ in (\ref{eq:define_w}) that
\begin{equation}
\tanh(w)=\frac{1-\frac{1}{4}\epsilon^2}{1+\frac{1}{4}\epsilon^2} \text{ \ \ and \ \ } \frac{\epsilon}{2}\cosh(w)=\frac{\cosh(w)}{e^w}=\frac{1}{2}\left(1+\frac{1}{4}\epsilon^2\right).
\label{eq:simplify_tanhw}
\end{equation}
Note that the right hand side in \eqref{eq:simplify_tanhw} is equal to the length of $\beta$, given that $\beta$ is equidistant to $\alpha_3$ at distance $w$.\\
For the following calculations we abbreviate
\begin{equation}
a_1=-\cosh(\alpha_1), \text{ \ \ } a_2=\cosh(\alpha_2) \text{ \ and \ } a=\cosh(\alpha_1) + \cosh(\alpha_2).
\label{eq:abbreviate_coshalpha}
\end{equation}
The curve $\tilde{\beta}= F_\beta \circ \beta = (\tilde{\beta}_1,\tilde{\beta}_2)$ with endpoints $B'_1$ and $B'_2$ (see \textit{Figure 3}) is not a horocycle, but almost: in Section \ref{sec:estimates} we shall show
\begin{lem} $\tilde{\beta}$ is the graph of a function $f:[a_1,a_2] \to \R$ with the following properties:
\begin{itemize}
\item[(i)] $2a(1-\frac{1}{8} \epsilon^2) \leq f(x) \leq 2a(1+\frac{1}{6}\epsilon^2)$
\item[(ii)]  $|f'(x)| \leq \frac{4}{15} \epsilon^2$.
\end{itemize}
Furthermore
\begin{itemize}
\item[(iii)]  $2a(1-\frac{1}{2} \epsilon^2) \leq \tilde{\beta}'_1(s) \leq 2a(1 + \frac{1}{6}\epsilon^2)$.
\end{itemize}
Here $\tilde{\beta}'_1(s)$ is the derivative of $\tilde{\beta}_1$ with respect to $s$ and $f'(x)$ is the derivative of $f$ with respect to $x$.
\label{thm:lemma1}
\end{lem}
We now adjust the mapping $F_\beta$ so that the image of $\mathcal{H}_\beta$ becomes $\mathcal{H}'_{2a}$, where
\begin{equation*}
\mathcal{H}'_{2a} = \{ z \in \mathcal{H}' \mid \im(z) \leq 2a \}.  \label{*eq:define_Hprime2a}
\end{equation*}
To this end we define $G_\beta: F_\beta(\mathcal{H}_\beta) \rightarrow \mathcal{H}'_{2a}$ as follows, using real notation
\begin{equation}
G_\beta(x,y) := \left\{ {\begin{array}{*{20}c}
    (x,y)  \\
    (x,\varphi(x,y))  \\
\end{array}} \right.\text{ \ if \ }
\begin{array}{*{20}c}
   y \leq a \\
   a \leq y \leq f(x), \\
\end{array}
\text{ \ where \ } \varphi(x,y) = a \left(1+ \frac{y-a}{f(x)-a}\right).
\label{eq:defGbeta}
\end{equation}
We remark that by \eqref{eq:abbreviate_coshalpha} the points $A'_1$, $i$, $A'_2$ and hence, the sides $\alpha_1$, $c'$, $\alpha_2$ of $\mathcal{H}'$ in $\Hy$ lie below the horizontal line $y = a$. These sides are therefore not affected by $G_{\beta}$.

For the points $(x,y)$ with $y \geq a$ the Jacobian matrix of $G_\beta$ is of the form
\begin{equation}
J_{G_\beta} =
 \left( {\begin{array}{*{20}c}
   {1} & {\frac{\partial \varphi}{\partial x}}  \\
   {0} & {\frac{\partial \varphi}{\partial y}}  \\
\end{array}} \right)
=
 \left( {\begin{array}{*{20}c}
   {1} & {\rho}  \\
   {0} & {1+\sigma}  \\
\end{array}} \right).
\label{eq:JGb}
\end{equation}
With \textbf{Lemma \ref{thm:lemma1}} a direct calculation yields the bounds $|\rho| \leq \frac{1}{3} \epsilon^2$ and $|\sigma| \leq \frac{1}{3} \epsilon^2$. It follows (see \textbf{Lemma \ref{thm:lemma2}}) that $G_\beta$ is quasiconformal with dilatation $q_{G_\beta}  \leq \frac{1}{1-\frac{\sqrt{2}}{{3}} \epsilon^2}$. By (\ref{eq:dilatation_estimate}) the corrected mapping $G_\beta \circ F_\beta: \mathcal{H}_\beta \rightarrow \mathcal{H}'_{2a}$ is a quasiconformal homeomorphism with dilatation
\begin{equation}
q_{G_\beta \circ F_\beta} \leq 1+ \epsilon^2.
\label{eq:dilatation_phibeta}
\end{equation}

The image of $\beta$ under this mapping is now the horizontal straight line at height $2a$ (the top side of $\mathcal{H}'_{2a}$, \textit{Figure} 3) but it does not have constant speed. We therefore need a second adjustment $H_{\beta} : \mathcal H'_{a,2a} \to \mathcal H'_{a,2a}$, where $\mathcal H'_{a,2a}$ is the rectangle 
\[
\mathcal H'_{a,2a} = \{ z \in \mathcal{H}' \mid a \leq \im(z) \leq 2a \} = \{x + iy \in \Hy \mid a_1 \leq x \leq a_2,\ a \leq y \leq 2a\}  \label{eq:define_Hprime_a2a}
\]
(see \eqref{eq:abbreviate_coshalpha} and the remark after \eqref{eq:defGbeta}).
In order to describe this mapping more conveniently \linebreak we apply an affine parameter change to $\beta$ and its image curves by replacing the  interval \linebreak $[-\epsilon_1\cosh(w),\epsilon_2 \cosh(w)]$ (see (\ref{eq:beta_unit_speed1})) by $[a_1,a_2]$. This allows us to write
\[
G_\beta \circ F_\beta(\beta(t))=(b_1(t),2a), t \in [a_1,a_2],
\]
where $b_1(a_1)=a_1$ and $b_1(a_2)=a_2$ and by \textbf{Lemma \ref{thm:lemma1}}(iii)
\begin{equation}
\epsilon \cosh(w)(1-\frac{1}{2}\epsilon^2) \leq b'_1(t) \leq \epsilon \cosh(w)(1+\frac{1}{6} \epsilon^2)
\label{eq:derivative_b1t}
\end{equation}
 with $\epsilon \cosh(w) = 1 + \frac{1}{4} \epsilon^2$ (see (\ref{eq:simplify_tanhw})). It is also more convenient to begin with the \textit{inverse} mapping $L_{\beta} : \mathcal H'_{a,2a} \to \mathcal H'_{a,2a}$ which we define as follows 
 \[
 L_{\beta}(x,y):=(x',y), \text{ \ where \ } x'= x + \frac{y-a}{a}(b_1(x)-x), \text{ \ for \ } (x,y) \in  \mathcal{H}'_{a,2a}.
\]
Using \eqref{eq:derivative_b1t} we see that $L_{\beta}$ is a diffeomorphism. The Jacobian matrix of $L_\beta$ is of the form
\begin{equation}
J_{L_\beta}=
\left( {\begin{array}{*{20}c}
   {1+ \sigma} & {0}  \\
   {\rho} & {1}  \\
\end{array}} \right), \text{ \ where \ }
\sigma= \frac{y-a}{a}(b'_1(x) -1), \;\rho = \frac{(b_1(x)-x)}{a} .
\label{eq:JacobLbeta}
\end{equation}
Using (\ref{eq:derivative_b1t}) we obtain, after a straightforward calculation, $|\sigma| \leq \frac{1}{2} \epsilon^2$, $ |\rho|  \leq  \frac{1}{4} \epsilon^2$. By \textbf{Lemma \ref{thm:lemma2}} again, the dilatation of $L_\beta$ has the upper bound $q_{L_\beta} \leq 1+ \frac{2}{3} \epsilon^2$.\\ 
The inverse mapping $H_{\beta} :=L_{\beta}^{-1}$ is a quasi-conformal homeomorphism of $\mathcal{H}'_{a,2a}$ onto itself with the same dilatation that brings the image of $\beta$ to constant speed (with respect to either, the euclidean and the hyperbolic metric). We extend $H_{\beta}$ to all of $\mathcal H'_{2a}$ by letting it be the identity below $\mathcal{H}'_{a,2a}$. The mapping 
\begin{equation}
\phi_{\beta} := H_{\beta} \circ G_{\beta} \circ F_{\beta} : \mathcal{H}_{\beta} \to \mathcal H'_{2a}
\label{eq:newphisubbeta}
\end{equation}
is now boundary coherent along $\beta$, and together with \eqref{eq:dilatation_phibeta} it follows that it is a quasiconformal homeomorphism with dilatation
\begin{equation}
q_{\phi_{\beta}} \leq 1+ 2\epsilon^2. 
\label{eq:dilatation_phibeta2}
\end{equation}

\bigskip
Let us now turn to the part above $\beta$. Here we want to use a conformal mapping $F^\beta$ that maps equidistant curves of side $\alpha_3$ to horocycles. \\
\\
We construct $F^\beta$ as follows. First, we use the M{\"o}bius transformation
\[
z \mapsto n(z)= e^{\epsilon_1} \frac{z+e^\lambda}{-z+e^\lambda}, z \in \Hy
\]
sending $\mathcal{H}^\beta$ to the domain $\Omega$ shown in \textit{Figure 4}. This mapping is actually a hyperbolic isometry.\\

\begin{figure}[h!]
\SetLabels
\L(.36*.02) $-1$\\
\L(.49*.02) $\,0$\\
\L(.62*.02) $1$\\
\L(.50*.17) $\,W$\\
\L(.59*.20) $\hat{B}_1$\\
\L(.62*.29) $\,n(\beta)$\\
\L(.39*.37) $\,\Gamma$\\
\L(.70*.40) $\hat{B}_2$\\
\L(.50*.51) $\,\,i$\\
\L(.58*.55) $\Omega$\\
\L(.50*.81) $\,\,ie^{\epsilon /2}$\\
\endSetLabels
\AffixLabels{%
\centerline{%
\includegraphics[height=5cm,width=8cm]{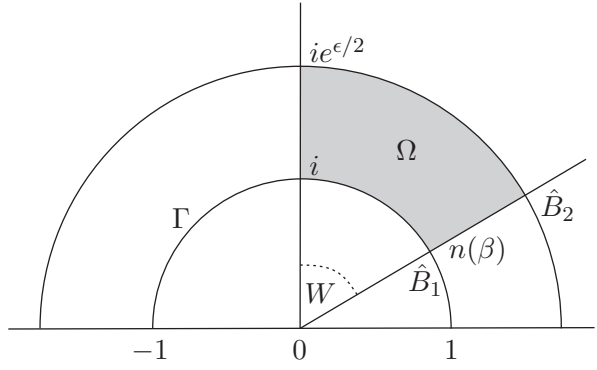}}}
\caption{The domain $\Omega$ in the upper half plane $\Hy$.}
\label{fig:figure4}
\end{figure}
Side $\alpha_3$ is sent to the imaginary axis with endpoints $E_1$ and $E_2$ going to $i$ and $i e^{\frac{\epsilon}{2}}$, respectively. The image $n(\beta)$ with endpoints $\hat{B}_1$ and $\hat{B}_2$ lies on the straight line through the origin that forms an angle $W$ with the imaginary axis. $\hat{B}_1$ lies on the geodesic $\Gamma$ at distance $w$ from $i=n(E_1)$. By (\ref{eq:define_Gamma}) we have
\begin{equation}
\sin(W) = \tanh(w) \text{ \ \ and \ \ } \cos(W) = \frac{1}{\cosh(w)}.
\label{eq:sincosW}
\end{equation}
The function $\zeta \mapsto \log(-i\zeta), \zeta \in \Omega$, maps $\Omega$ conformally to the rectangle
\[
   \{ \zeta' \in \C \mid 0 \leq \re(\zeta') \leq \frac{1}{2} \epsilon, -W \leq \im(\zeta') \leq 0 \}.
\]
We thus define, for $z \in \mathcal{H}^\beta$
\begin{equation}
    F^\beta (z) := \frac{2a}{\epsilon} \log(-i n(z)) + a_1 + 2ia(1+ \frac{W}{\epsilon}).
\label{eq:Fupperbeta}
\end{equation}
$F^\beta$ sends $\mathcal{H}^\beta$ to the rectangle
\[
\hat{\mathcal{H}}'^{2a} = \{ \zeta \in \mathcal{H}' \mid 2a \leq \im(\zeta) \leq 2a(1+\frac{W}{\epsilon}) \}.
\]
The part of $\mathcal{H}'$ above $\mathcal{H}'_{2a}$, however, that corresponds to $Y^{\epsilon^*}_{l_1,l_2}$, is
\[
 \mathcal{H}'^{2a} := \{ \zeta \in \mathcal{H}' \mid 2a \leq \im(\zeta) \leq \frac{2a}{\epsilon^*} \}
\]
whose top side has hyperbolic length $\frac{1}{2} \epsilon^*$ with $\epsilon^* = \frac{2}{\pi}\epsilon$, see \textbf{Theorem~\ref{thm:main}}.  We apply therefore a `vertical correction'
\[
   G^\beta(x,y) := (x,\frac{1}{k_\epsilon}(y-2a)+2a), (x,y) \in \hat{\mathcal{H}}'^{2a} , \text{ \ where \ }
   k_\epsilon = \frac{2W}{\pi-2\epsilon} = \frac{1}{\frac{\pi}{2}-\epsilon} \arccos\left(\frac{\epsilon}{1+\frac{\epsilon^2}{4}}\right).
\]
Here the expression on the right hand side follows from (\ref{eq:simplify_tanhw}) and (\ref{eq:sincosW}). Setting
\begin{equation}
\phi^{\beta} :=G^\beta \circ F^\beta 
\label{eq:phiupperbeta}
\end{equation}
we now have $\phi^{\beta}(\mathcal{H}^\beta) = \mathcal{H}'^{2a}$ as required. Furthermore, since $F^\beta$ is conformal $\phi^\beta$ and $G^\beta$ have the same dilatation $k_\epsilon$. An elementary consideration shows that
\begin{equation}
1 \leq k_\epsilon \leq 1+ \frac{\epsilon^3}{6\pi}(1+\epsilon) < 1+ \frac{1}{25} \epsilon^2.
\label{eq:dilatation_kepsilon}
\end{equation}
One easily checks that $\phi^{\beta} \circ \alpha_3$ and $\phi^{\beta} \circ \beta$ have constant Euclidean and hyperbolic speeds. In  particular $\phi^{\beta}$ is boundary coherent along $\alpha_3$. Moreover, it matches with $\phi_\beta$ along the common side $\phi_\beta(\beta) = \phi^{\beta}(\beta)$. We thus set $\phi = \phi_\beta$ on $\mathcal{H}_\beta$ and $\phi = \phi^\beta$ on $\mathcal{H}^\beta$ and extend it symmetrically to the back side $\tilde{\mathcal{H}}$ of $Y_{l_1,l_2,\epsilon}$. The mapping $\phi$ thus defined is boundary coherent and by  \eqref{eq:dilatation_phibeta2}, \eqref{eq:phiupperbeta} and \eqref{eq:dilatation_kepsilon} has dilatation $\leq 1 + 2\epsilon^2$. This completes the proof of \textbf{Theorem \ref{thm:main}} modulo the estimates of the next section.

\section{Estimates}  \label{sec:estimates}
In this section we provide the missing proofs of inequality (\ref{eq:dilatation_estimate}) and \textbf{Lemma \ref{thm:lemma1}}. As the trigonometric functions involved in the construction of $\phi$ lead to unmanageable expressions our approach is through inequalities that include small numerical constants.\\
The first such inequality is the following, where $\lambda$ is the height of the hexagon $\mathcal{H}$ (see \textit{Figure 2}).
\begin{equation}
1 \le \frac{e^\lambda}{2\sinh(\lambda)} \leq 1+ \frac{1}{63} \epsilon^2.
\label{eq:lambda_epsilon}
\end{equation}
\begin{proof}
We recall that by hypothesis $0 \leq \epsilon \leq \frac{1}{2}$ and $ \epsilon_1 + \epsilon_2 = \frac{1}{2}\epsilon$. Let $\epsilon_k$ be the smaller of $\epsilon_1$ and $\epsilon_2$. Then with (\ref{eq:pent_formulas})
\[
\frac{1}{2} e^\lambda \sinh(\frac{\epsilon}{4}) \geq \sinh(\lambda)\sinh(\epsilon_k) = \cosh(\alpha_k) > 1.
\]
Hence
\[
 2e^{-\lambda} \sinh(\lambda) = 1- e^{-2\lambda} \geq 1- \frac{1}{4} \sinh(\frac{\epsilon}{4})^2,
\]
from which the estimate in (\ref{eq:lambda_epsilon}) follows .
\end{proof}
Another such inequality is
\begin{equation}
1 > \sigma_1 :=\frac{\frac{1}{2}\epsilon}{\sinh(\epsilon_1)+ \sinh(\epsilon_2)} \geq \frac{\frac12 \epsilon}{\sinh(\frac12 \epsilon)} \geq 1-\frac{\epsilon^2}{24}.
\label{eq:sigma1_epsilon}
\end{equation}
For the altitude $e^{\lambda-w}$ of the point $\beta(0)$ on the upper half plane (see \textit{Figure 3} and \eqref{eq:beta_unit_speed1}) and the altitude $2a$ of $\phi_\beta(\beta)$ we have
\begin{eqnarray}
e^{\lambda-w} = 2a \cdot \sigma_2, \text{ \ with \ } \sigma_2&:=&\frac{e^\lambda}{2\sinh(\lambda)} \cdot \sigma_1 \text{ \ satisfying \ } \nonumber \\
1-\frac{\epsilon^2}{24} \leq \sigma_2 &\leq&  1+\frac{\epsilon^2}{63}. \label{eq:sigma2_epsilon}
\end{eqnarray}
\begin{proof}
This follows by combining (\ref{eq:lambda_epsilon}) and (\ref{eq:sigma1_epsilon}) and using that $\sinh(\lambda)(\sinh(\epsilon_1) + \sinh(\epsilon_2)) = \cosh(\alpha_1) + \cosh(\alpha_2)=a$ (see (\ref{eq:pent_formulas}),(\ref{eq:define_w}) and (\ref{eq:abbreviate_coshalpha})).
\end{proof}
For the endpoints of $\beta$ (again parametrized as in \eqref{eq:beta_unit_speed1}, \eqref{eq:beta_unit_speed2})
\[
B_i = \beta(s_i) =x_i + i y_i, \text{ \ \ where \ \ } s_1=-\epsilon_1 \cosh(w) \text{  \ and \ } s_2=\epsilon_2 \cosh(w)
\]

we have
\begin{eqnarray}
x_i= a_i \sigma_{x_i} \text{ \ with \ }  \sigma_{x_i}&:=& \frac{e^\lambda}{2\sinh(\lambda)} \cdot \frac{2}{\cosh(\epsilon_i) + \tanh(w)} \text{ \ \ satisfying \ } \nonumber \\
1+\frac{1}{6} \epsilon^2 \leq \sigma_{x_i} &\leq&  1+\frac{2}{7} \epsilon^2 \label{eq:sigmaxi_epsilon}
\end{eqnarray}
and
\begin{eqnarray}
y_i= 2a \sigma_{y_i} \text{ \ with \ }  \sigma_{y_i}&:=& \frac{\sigma_2}{(\cosh(\epsilon_i)-1)\frac{\cosh(w)}{e^w} +1}  \text{ \ \ satisfying \ } \nonumber \\
1-\frac{1}{8} \epsilon^2 \leq \sigma_{y_i} &\leq&  1+\frac{1}{63} \epsilon^2. \label{eq:sigmayi_epsilon}
\end{eqnarray}
\begin{proof}
This follows from the parametrization in (\ref{eq:beta_unit_speed2}) plus (\ref{eq:define_w}),(\ref{eq:simplify_tanhw}) and (\ref{eq:lambda_epsilon}) and elementary simplifications.
\end{proof}

We are now ready to prove the inequality (\ref{eq:dilatation_estimate}) for $q_{F_\beta}$.  First we show that $\eta_i= H_iD_0$, (see (\ref{eq:etai}) and \textit{Figures 2,3}) has a lower bound independent of $\epsilon$:
\begin{equation}
\eta_i \geq \frac{2}{5} \frac{\cosh(\alpha_i)}{(\cosh(\alpha_1)+ \cosh(\alpha_2))^2} = \frac{2}{5} \frac{|a_i|}{a^2}.
\label{eq:lowerbound_etai}
\end{equation}
\begin{proof}
By (\ref{eq:Fermi_Gamma1}) we have that $\eta_i > \tanh(\eta_i) = \frac{2|x_i|}{x^2_i + y^2_i + 1}$. By (\ref{eq:sigmaxi_epsilon}) and (\ref{eq:sigmayi_epsilon}) using that $\sigma^2_{x_i} + 4 \sigma^2_{y_i} \leq 5(1+\frac{1}{6} \epsilon^2)$ we obtain
\begin{equation}
x^2_i + y^2_i + 1 \leq 5(1+\frac{1}{6} \epsilon^2)a^2.
\label{eq:xiyi_epsilon}
\end{equation}
Together with the lower bound in (\ref{eq:sigmaxi_epsilon}) this proves (\ref{eq:lowerbound_etai}).
\end{proof}
In (\ref{eq:displace_estimate}) we have shown that $\delta_i \leq \frac{\sinh(\epsilon_i)^2}{2|a_i|}$. By (\ref{eq:sigma1_epsilon}), (\ref{eq:pent_formulas}) and the definitions in \eqref{eq:abbreviate_coshalpha} we have  $\sinh(\epsilon_i) \cdot \frac{2}{\epsilon} = \frac{|a_i|}{a}\cdot \frac{1}{\sigma_1}$. Hence,
\begin{equation}
\delta_i \leq \frac{\epsilon^2}{8 \sigma_1^2}\cdot \frac{|a_i|}{a^2} < \frac{\epsilon^2}{16 \sigma_1^2}.
\label{eq:upperbound_deltai}
\end{equation}
Together with (\ref{eq:lowerbound_etai}) and \eqref{eq:sigma1_epsilon} this yields
\begin{equation}
\frac{\delta_i}{\eta_i} \leq \frac{1}{3} \epsilon^2.
\label{eq:upperbound_deltaietai}
\end{equation}
This completes the proof of inequality (\ref{eq:dilatation_estimate}). \hfill $\square$ \\

Later we shall also need the following upper bound
\begin{equation}
\tanh(\eta_i) = \frac{2 \vert a_i \vert \sigma_{x_i}}{a_i^2 \sigma_{x_i}^2+4a^2 \sigma_{y_i}^2+1}\leq\frac19 + \frac{1}{16}\epsilon^2.
\label{eq:upperbound_thetai}
\end{equation}
The bound is obtained via \eqref{eq:sigmaxi_epsilon}, \eqref{eq:sigmayi_epsilon} using that $a \geq 1 + \vert a_i \vert$.\\

We turn to the proof of \textbf{Lemma \ref{thm:lemma1}} using the notation
\begin{eqnarray*}
    \beta(s) &=& \beta_1(s) + i \beta_2(s) \\
    \tilde{\beta}(s) &=& F_\beta(\beta(s)) =  \tilde{\beta}_1(s) + i \tilde{\beta}_2(s).
\end{eqnarray*}

To avoid distinctions of cases with different signs we restrict ourselves without loss of generality to $s \geq 0$. Thus, $s \in [0,\epsilon_2 \cosh(w)]$.\\
As $\beta$ is a half-circle we have by (\ref{eq:sigmayi_epsilon}) for the lower bound and by \eqref{eq:beta_unit_speed1}, \eqref{eq:sigma2_epsilon} for the upper
\begin{equation}
2a(1-\frac{1}{8} \epsilon^2) \leq \beta_2(s) \leq \beta_2(0) = e^{\lambda - w}\leq 2a(1+ \frac{1}{63} \epsilon^2).
\label{eq:beta2_epsilon}
\end{equation}
For the tangent vectors $\beta'(s)$ the Euclidean norm is denoted by $|\beta'(s)|$ and the hyperbolic norm by $\|\beta'(s) \| = \frac{|\beta'(s)|}{\beta_2(s)}$. The slope of $\beta$ for $s \in [0,\epsilon_2 \cosh(w)]$ satisfies
\begin{equation}
 -s \frac{\epsilon^2}{2}(1-\frac{5}{16} \epsilon^2) \geq \frac{\beta'_2(s)}{|\beta'(s)|} = \frac{\beta'_2(s)}{\beta_2(s)} \geq -\frac{\epsilon}{2} \sinh(\frac{\epsilon}{2}).
 \label{eq:beta2_epsilon1}
\end{equation}
\begin{proof}
Since $\| \beta'(s)\| = 1$ (see below \eqref{eq:define_x}) we have the equation in the middle. By  (\ref{eq:beta_unit_speed2}) and using that $\kappa = \frac{1}{\cosh(w)}$ \eqref{eq:define_x} and $e^{-w} = \frac{\epsilon}{2}$ \eqref{eq:define_w} we have
\[
\frac{\beta'_2(s)}{\beta_2(s)} = \frac{- \frac{\epsilon}{2} \sinh(\kappa s)}{1+ \frac{\cosh(w)}{e^w}(\cosh(\kappa s) -1)}.
\]
The inequalities then follow from a straightforward elementary calculation via \eqref{eq:simplify_tanhw} using for the first inequality that $\cosh(\kappa s) \leq \cosh(\epsilon_2) \leq \cosh(\frac{\epsilon}{2})$ and $\sinh(\kappa s) \leq \sinh(\frac{\epsilon}{2})$ for the second.
\end{proof}
Since $\Vert \beta'(s) \Vert =1$ and hence $\beta_1'(s)^2 + \beta_2'(s)^2 = \beta_2(s)^2$, Equation \eqref{eq:beta2_epsilon1} implies that $\beta_1'(s)^2 \geq \beta_2(s)^2(1  -(\frac{\epsilon}{2} \sinh(\frac{\epsilon}{2}))^2)$. From this we get, by \eqref{eq:beta2_epsilon1} again, and after elementary simplification
\begin{equation}
 \frac{\beta'_2(s)}{\beta'_1(s)} \geq -\frac{4}{15} \epsilon^2.
\label{eq:beta2beta1_epsilon}
\end{equation}
To prove similar estimates for $\tilde{\beta}$ we look at the action of $F_\beta$ in $\Hy$.\\
To this end we fix $s \in [0, \epsilon_2 \cosh(w)]$ and consider the hyperbolic isometry $g$ with axis $\Gamma$ that moves the point $\beta(s)$ to the point $F_\beta(\beta(s))$. By (\ref{eq:Fbeta1}) $g$ has oriented displacement length

\begin{equation}
\delta = \delta(s) = - \frac{\delta_2}{\eta_2} t(s), \quad t(s) \in [0,\eta_2],
\label{eq:deltas}
\end{equation}
where by (\ref{eq:Fermi_Gamma1})
\begin{equation}
\tanh(t(s)) = \frac{2\beta_1(s)}{\beta^2_1(s) + \beta^2_2(s) + 1}.
\label{eq:tanh_ts}
\end{equation}
In complex notation $g$ is given as $g(z) = \rgh(\delta,z), z \in \Hy$ (see (\ref{eq:define_up_rgh})). For any tangent vector $v=(v_1,v_2) = v_1 + i v_2$ at $z$ the tangent map $dg$ acts by multiplication:
\begin{equation}
dg(v) = g'(z) \cdot v,
\label{eq:dg_v}
\end{equation}
where $g'$ is the complex derivative
\begin{equation}
g'(z) = \frac{1}{(\sinh(\frac{\delta}{2})z + \cosh(\frac{\delta}{2}))^2}.
\label{eq:complex_dg}
\end{equation}
Let now $z$ be the point $\zeta= \beta(s)$. Both $F_\beta$ and $g$ send $\zeta$ to the point
\[
   \tilde{\beta}(s) = F_\beta(\beta(s)) = g(\zeta).
\]
Since $g$ shifts $\zeta$ towards the imaginary axis along the equidistant line of $\Gamma$ i.e. the half circle in $\Hy$ through $-1,1$ and $\zeta$ we have using \eqref{eq:beta2_epsilon}
\begin{equation}
 \tilde{\beta}_2(s) \geq \beta_2(s) \geq 2a(1-\frac{\epsilon^2}{8}) 
\label{eq:lowbdbtilde}
\end{equation}
For an upper bound of $\tilde{\beta}_2(s)$ we use that
\begin{eqnarray*}
   \im(g(\zeta)) &=& \im(\zeta) \cdot \left( \cosh(\frac{\delta}{2})^2 + \sinh(\frac{\delta}{2})^2 |\zeta|^2 + \sinh(\delta) \re(\zeta) \right)^{-1} \\
                    &\leq& \im(\zeta)\cdot \left( 1-|\sinh(\delta) \re(\zeta)|\right)^{-1},
\end{eqnarray*}
where $\vert \delta \vert \leq \delta_2$ (see \eqref{eq:deltas}).
By (\ref{eq:sigmaxi_epsilon}) and (\ref{eq:upperbound_deltai}) we have, recalling that $\re(\zeta) = \beta_1(s)$,
\[
   |\sinh(\delta) \re(\zeta)| \leq \frac{\sinh(\delta_2)}{\delta_2} \delta_2 x_2 \leq \frac{\sinh(\delta_2)}{\delta_2} \cdot \frac{\epsilon^2}{8} \cdot \frac{\sigma_{x_2}}{\sigma^2_1} < 1.1 \cdot \frac{\epsilon^2}{8}.
\]
Using that $\im(\zeta) = \beta_2(s)$ and $\tilde{\beta}_2(s) = \im(g(\zeta))$ we get
\begin{equation}
\beta_2(s) \leq \tilde{\beta}_2(s) \leq (1+ \frac{1}{7} \epsilon^2) \beta_2(s).
\label{eq:beta2_tildebeta2}
\end{equation}
Together with \eqref{eq:lowbdbtilde} this proves item (i) of \textbf{Lemma \ref{thm:lemma1}}.

For the slope of $\tilde{\beta}$ we estimate the angle between the tangent vectors $\beta'(s)$ and $\tilde{\beta}'(s)$. From the definition of $F_\beta$ in  (\ref{eq:Fbeta1}) it follows that the tangent map $dF_\beta$ is a product
\begin{equation}
dF_\beta = dg \cdot A,
\label{eq:dFbeta_dgA}
\end{equation}
where $A$ is a linear map in the tangent space of $\Hy$ at $\zeta=\beta(s)$ that, with respect to a suitable orthonormal basis, has the matrix
\[
M_A=
\left( {\begin{array}{*{20}c}
   {1- \frac{\delta_2}{\eta_2}} & {0}  \\
   {0} & {1}  \\
\end{array}} \right).
\]
This fact is seen in the Fermi coordinates used in (\ref{eq:Fbeta1}) for which the metric tensor (\ref{eq:Fermi_metric}) is independent of $t$; the underlying basis vectors for $A$ are tangent to the coordinate lines. It is an elementary exercise to show that for such a mapping the angle $\vartheta_A = \measuredangle(\beta'(s),A\beta'(s))$ between $\beta'(s)$ (or any other tangent vector) and its image satisfies
\begin{equation*}
|\sin(\vartheta_A)| \leq \frac{\frac{\delta_2}{\eta_2}}{2-\frac{\delta_2}{\eta_2}}.
\end{equation*}
A glance at \textit{Figure} 2 shows that for $s \geq 0$ (to which case we have restricted ourselves) the angle $\vartheta_A$ between $\beta'(s)$ and $A \beta'(s)$ is $\geq 0$. Furthermore, by \eqref{eq:upperbound_deltaietai}, $\frac{\delta_2}{\eta_2} \leq \frac13 \epsilon^2$. Hence, after some simplification,
\begin{equation}
0 \leq \vartheta_A \leq j \cdot \frac{\epsilon^2}{6}
\label{eq:sin_varthetaA}
\end{equation}
with $j= 1.05$.\\
Note that $\vartheta_A$ is also the Euclidean angle between the vectors because the model $\Hy$ of the hyperbolic plane is conformal. \\
\\
For the Euclidean angle $\vartheta_g := \measuredangle(A\beta'(s),dg(A\beta'(s)))$ we have by (\ref{eq:dg_v}) and (\ref{eq:complex_dg}) applied to $z = \zeta = \beta(s)$
\begin{equation}
 \vartheta_g = \arg(g'(\zeta))= -2 \arg\left(\sinh(\frac{\delta(s)}{2})\zeta + \cosh(\frac{\delta(s)}{2})\right),
\label{eq:varthetag}
\end{equation}
where $\delta(s)$ is given by (\ref{eq:deltas}). By \eqref{eq:tanh_ts}, the monotonicity of the function $\tau \mapsto \frac{\tau}{\tanh(\tau)}$ and using \eqref{eq:upperbound_thetai}  we get
\[
t(s) < j_1 \frac{2 \beta_1(s)}{\beta_1(s)^2 + \beta_2(s)^2}
\]
with $j_1 = \frac{\eta_2}{\tanh(\eta_2)} < 1.01$. Furthermore, by \eqref{eq:beta_unit_speed1} $\beta_1(s) = \sinh(\kappa s) \cosh(\omega) \beta_2(s)$, where $\kappa s \leq \epsilon_2 \leq \frac{\epsilon}{2}$. Recalling that $\kappa = \frac{1}{\cosh(w)}$ \eqref{eq:define_x}, we get
\[
s \beta_2(s) \leq \beta_1(s) \leq j_2 \cdot s \beta_2(s),
\]
with $j_2 = 4 \sinh(\frac14)< 1.02$, and from this
\[
t(s) < j_1 j_2 \frac{2s \beta_2(s)}{\beta_1(s)^2 + \beta_2(s)^2} \leq j_1 j_2 \frac{2s}{\sqrt{s^2+1}}\cdot \frac{1}{\vert \beta(s) \vert}.
\]
Setting $j_3 = \frac{2}{\delta_2} \sinh(\frac{\delta_2}{2})$ we conclude using \eqref{eq:tanh_ts} and then \eqref{eq:upperbound_deltaietai}
\[
\vert \sinh(\frac{\delta(s)}{2}) \beta(s) \vert \leq j_3 \vert \frac{\delta(s)}{2}\beta(s) \vert = j_3 \cdot \frac12 \cdot \frac{\delta_2}{\eta_2}\vert t(s) \beta(s) \vert \leq j_1 j_2 j_3 \frac{s}{\sqrt{s^2+1}} \cdot \frac{\epsilon^2}{3}.
\]
By \eqref{eq:upperbound_deltai} $j_3 < 1.001$. The second summand in the argument as of  \eqref{eq:varthetag} is a real number $>1$. Recalling that for $s \geq 0$ we have $\delta(s) \leq 0$ and hence, $\vartheta_g \geq 0$, we obtain therefore the following
\begin{equation}
0\leq \vartheta_g \leq j \cdot \frac{2\epsilon^2}{3} \frac{s}{\sqrt{s^2+1}},
\label{eq:estimate_varthetag}
\end{equation}
where $j$ includes all the preceding correcting factors and has value $j \leq 1.05$. \\
By (\ref{eq:dFbeta_dgA}), \eqref{eq:sin_varthetaA} and \eqref{eq:estimate_varthetag} the tangent map $dF_\beta$ rotates the tangent vector $\beta'(s)$ by an angle $\vartheta_A + \vartheta_g \geq 0$, where $\vartheta_A$ and $\vartheta_g$ have the upper bounds (\ref{eq:sin_varthetaA}) and (\ref{eq:estimate_varthetag}), while the tangent vector $\beta'(s)$ satisfies (\ref{eq:beta2_epsilon1}). Bringing this together we obtain the following upper bound for the oriented angle $\vartheta(s)$ between $\tilde{\beta}'(s)$ and the horizontal line through $\tilde{\beta}(s)$
\[
\vartheta(s) \leq \epsilon^2 \left( -\frac{s}{2}(1-\frac{5}{16}\epsilon^2) +  \frac{2j}{3}\frac{s}{\sqrt{s^2+1}} + \frac{j}{6}
\right).
\]
The right hand side is a monotonically increasing function of $s$ and $0 \leq s \leq \frac{\epsilon}{2} \cosh(w) = \frac12(1+\frac14 \epsilon^2)$ (by \eqref{eq:simplify_tanhw}). It follows that $\vartheta(s) < \frac14 \epsilon^2$ and
\[
   \frac{\tilde{\beta}'_2(s)}{\tilde{\beta}'_1(s)} \leq \frac{4}{15} \epsilon^2.
\]
Since $dF_\beta$ rotates $\beta'(s)$ counterclockwise we have $\frac{\tilde{\beta}'_2}{\tilde{\beta}'_1} \geq \frac{\beta'_2}{\beta'_1}$. Hence, together with (\ref{eq:beta2beta1_epsilon}) we have proved item (ii) of \textbf{Lemma \ref{thm:lemma1}}. \\
As for item (iii), by \eqref{eq:dFbeta_dgA} the tangent vector $\tilde{\beta}'(s) = dF_{\beta}(\beta'(s))$ has hyperbolic length
\[
1 - \frac{\delta_2}{\eta_2} \leq \Vert \tilde{\beta}'(s) \Vert = \Vert A(\beta'(s))\Vert \leq 1.
\]
owing to the fact that $\Vert \beta'(s) \Vert = 1$. Since by \textbf{Lemma \ref{thm:lemma1}}(ii) $\tilde{\beta}'(s)$ satisfies $\left \vert\frac{\tilde{\beta}'_2(s)}{\tilde{\beta}'_1(s)}\right \vert \leq \frac{4}{15} \epsilon^2$ and as $\vert \tilde{\beta}'(s)\vert = \tilde{\beta}_2(s) \Vert\tilde{\beta}'(s)\Vert$ item (iii) now follows from \eqref{eq:upperbound_deltaietai}, \eqref{eq:beta2_epsilon}, \eqref{eq:beta2_tildebeta2} by elementary simplifications.\\

We finish this section with two known estimates for linear maps which we prove for convenience. For the purpose of this paper a piecewise smooth mapping $F : M \to N$ from a Riemannian manifold $M$ into a Riemannian manifold $N$ is said to have \emph{length distortion} $\leq k$ if for any tangent vector $v$ of $M$
\[
\frac{1}{k}  \Vert v \Vert_M \leq \Vert dF(v) \Vert_N \leq k \cdot \Vert v \Vert_M,
\]
where $\Vert \, . \,  \Vert_M$ and $\Vert \, . \,  \Vert_N$ are the respective Riemannian norms. 

In the following lemma the Riemannian metric in question is the standard Euclidean metric of $\R^2$.

\begin{lem} Consider the linear mapping $\psi : \R^2 \rightarrow \R^2$ defined by the matrix
\[
M=
\left( {\begin{array}{*{20}c}
   {1} & {\rho}  \\
   {0} & {1+ \sigma}  \\
\end{array}} \right),
\text{ \ \ where \ \ }  \rho^2 + \sigma^2 < 1.
\]
Then $\psi$ is quasiconformal with dilatation $q= \frac{1}{1-\sqrt{\rho^2 + \sigma^2}}$ and has length distortion $q' \leq \frac{1+\frac12 \vert \rho \vert + \frac18 \rho^2}{1-|\sigma|}$.
\label{thm:lemma2}
\end{lem}
\begin{proof}
For the first statement we follow \cite{alp1}, \textbf{Lemma 3.2} arguing with the Beltrami coefficient 
\[
\mu= \frac{\partial \psi}{ \partial \bar{z}}/\frac{\partial \psi}{ \partial z}, \text{ \ \ where \ \ }
\frac{\partial}{\partial \bar{z}} = \frac{1}{2}\left( \frac{\partial}{\partial x} + i  \frac{\partial}{\partial y}\right) \text{ \ and \ } \frac{\partial}{\partial z} = \frac{1}{2}\left( \frac{\partial}{\partial x} - i  \frac{\partial}{\partial y}\right).
\]
In our case $\mu$ becomes $\mu= \frac{-\sigma + i \rho}{2+\sigma -i\rho}$ and the dilatation is $q= \frac{1+|\mu|}{1-|\mu|}$. For given value $d$ of $\sqrt{\rho^2 + \sigma^2}$ the value of $|\mu|$ is maximal when $\rho = 0$ and $\sigma = -d$, and the first statement follows.

For the second statement we identify $M$ with $d\psi$ and decompose $M$ into a product
\[
M=M_{\sigma}M_{\rho}, \quad \text{with} \quad 
M_{\sigma}=\left( {\begin{array}{*{20}c}
   {1} & {0}  \\
   {0} & {1+ \sigma}  \\
\end{array}} \right),\;
M_{\rho} = \left( {\begin{array}{*{20}c}
   {1} & {\rho}  \\
   {0} & {1
  }  \\
\end{array}} \right)
\]
For the operator norms  we get $\max\{\vert M_{\sigma}\vert, \vert M^{-1}_{\sigma}\vert\}  = \frac{1}{1-\vert \sigma \vert}$, $\vert M_{\rho} \vert^2 = \vert M^{-1}_{\rho} \vert^2 = 1 + \vert \rho \vert\sqrt{1 + \frac14 \rho^2} + \frac12 \rho^2$ and the bound follows by elementary simplification.
\end{proof}

\section{Concluding remarks}\label{sec:concluding}

We finish with a refinement of \textbf{Theorem~\ref{thm:main}} and some further corollaries.

Let again $Y =Y_{l_1,l_2,l_3}$ be as in Section~\ref{sec:results}. For $i = 1,2,3$ we have the collars $\mathcal{C}_i$ as in \eqref{eq:collars}. In addition to this we define the \emph{reduced collars} 
\begin{equation}
   \hat{\mathcal{C}}_i = \{ p \in Y \mid \dist(p,\gamma_i) < w_i \} , \text{\ where\ } w_i:=\log \left(\frac{2}{l_i}\right),
\label{eq:reducedC}
\end{equation}
(assuming $l_i > 0$), where we recall that $l_i = \ell(\gamma_i)$. If $l_i \geq 2$ then $\hat{\mathcal{C}}_i$ is simply the empty set. For $0 < l_i < 2$ the width $w_i$ is smaller than the width $\omega_i$ of $\mathcal{C}_i$. In fact, by \eqref{eq:collars} we even have
\begin{equation}
\sinh(w_i + \log(2))  \sinh(\frac{l_i}{2}) = \frac12\left( \frac{4}{l_i}-\frac{l_i}{4} \right)\sinh(\frac{l_i}{2} )< 1=\sinh(\omega_i) \sinh(\frac{l_i}{2}).
\label{eq:wipluslog2}
\end{equation}
Hence, $\hat{\mathcal{C}}_i$ is a proper subset of $\mathcal{C}_i$. When $\ell(\gamma_i) < 2$ the boundary component of $\hat{\mathcal{C}}_i$ that lies in the interior of $Y$ is an equidistant curve of $\gamma_i$ of length
\[
   \hat{l}_i := l_i \cosh(w_i) = 1 + \frac14 l_i^2.
\]
We extend this by setting
\[
   \hat{l}_i := 1  \text{\ if\ } l_i = 0 \quad \text{and} \quad \hat{l}_i := l_i \text{\ if\ } l_i \geq 2.
\]
In the degenerate case where $l_i = 0$ and $\gamma_i$ is a puncture the reduced collar $\hat{\mathcal{C}}_i$ is defined as the set of all points that lie outside the horocycle of length 1 around the puncture. In this case the boundary components of $\hat{\mathcal{C}}_i$ and $\mathcal{C}_i$ in the interior of $Y$ are equidistant curves with distance $\log(2)$.\\
Cutting $\hat{\mathcal{C}}_1$, $\hat{\mathcal{C}}_2$, $\hat{\mathcal{C}}_3$ away we obtain the \emph{reduced Y-piece} (or `kernel')
\begin{equation}
  \hat{Y}_{l_1,l_2,l_3} = Y_{l_1,l_2,l_3}^{\hat{l}_1,\hat{l}_2,\hat{l}_3} = Y_{l_1,l_2,l_3} \setminus (\hat{\mathcal{C}}_1 \cup \hat{\mathcal{C}}_2 \cup \hat{\mathcal{C}}_3)
 \label{eq:reducedY}
\end{equation}
(see the notation introduced below \eqref{eq:uhp_metric}). We now have

\begin{thm}
The mapping $\phi : Y_{l_1,l_2,\epsilon} \rightarrow Y^{\epsilon^*}_{l_1,l_2}$ constructed in Section~\ref{sec:proofthm3} has the following properties.
\begin{enumerate}[$(i)$]
\item $\phi$ is boundary coherent and quasiconformal with dilatation $q_{\phi} \leq 1+ 2\epsilon^2$.
\item $\phi(\hat{Y}_{l_1,l_2,\epsilon}) = \hat{Y}_{l_1,l_2,0}$.
\item The restriction of $\phi$ to $\hat{Y}_{l_1,l_2,\epsilon}$ is boundary coherent and has length distortion $ \leq 1+\frac52 \epsilon^2$.
\item The restriction of $\phi$ to $\hat{\mathcal{C}}_i$ is an isometry for $i \in \{1,2\}$.
\end{enumerate}
\label{thm:mainrefined}
\end{thm}
\begin{proof}
(i) is the restatement of \textbf{Theorem~\ref{thm:main}}. We proceed with (iv) and assume that  $0 < l_i < 2$ (for $l_i \geq 2$ there is nothing to prove and for $l_i = 0$ we may argue by continuity).The task is to show that the mapping $F_{\beta}$ \eqref{eq:Fbeta1} acts isometrically on the set $\{ p \in \mathcal{H} \mid \dist(p, \alpha_i) \leq w_i \}$ and that its image under $F_{\beta}$ is untouched by the adjusting mappings $G_{\beta}$ and $H_{\beta}$ \eqref{eq:newphisubbeta}. As $F_{\beta}$ acts isometrically on the shaded area $B_iA_iC_iA_i$ (\textit{Figure 3}) and $G_{\beta}$ and $H_{\beta}$ act isometrically below the line $y = a$ it suffices  to show that $\dist(C_i,H_i) > w_i$ and the (hyperbolic) distance from $A'_i$ to the line $y=a$ is $> w_i$. Now, by \eqref{eq:wipluslog2} and \eqref{eq:pent_formulas} $\dist(C_i,D_0) > w_i + \log(2)$, while by \eqref{eq:upperbound_thetai} $\dist(H_i,D_0) = \eta_i < \log(2)$ which implies the first inequality. As for the second, by \eqref{eq:A1A2prime} $A'_i$ has imaginary part $\sinh(\alpha_i)$ and so the distance from $A'_i$ to the line $y = a = \cosh(\alpha_1) + \cosh(\alpha_2)$ is
\[
\log(a)-\log(\sinh(\alpha_i)) \geq \log(1+\cosh(\alpha_i))-\log(\sinh(\alpha_i)) = \arcsinh\left(\frac{1}{\sinh(\alpha_i)}\right)  =\omega_i > w_i.
\]
Statement (ii) and the boundary coherence in (iii) are immediate consequences of (iv), \eqref{eq:newphisubbeta} and the remark following \eqref{eq:newphisubbeta}. It remains to estimate the length distortion of $\phi_{\beta} =  H_{\beta} \circ G_{\beta} \circ F_{\beta}$.

By \eqref{eq:dilatation_estimate} and the remark preceding it $F_{\beta}$ has length distortion $\leq (1-\frac13\epsilon^2)^{-1}$. By \textbf{Lemma \ref{thm:lemma2}} and the estimates below \eqref{eq:JacobLbeta} $L_{\beta}$ and its inverse $H_{\beta}$ have length distortion $\leq 1 + \frac34 \epsilon^2$ with respect to the Euclidean metric. Since $\im(H_{\beta}(z)) = \im(z)$ for all $z$ in the domain of $H_{\beta}$ the same bound holds for the hyperbolic metric \eqref{eq:uhp_metric}. By \eqref{eq:JGb} and the estimates thereafter $G_{\beta}$ has Euclidean distortion $\leq 1 + \frac35 \epsilon^2$. By \textbf{Lemma \ref{thm:lemma1}}(i) the imaginary parts $y$ of $z$ and $y'$ of $G_{\beta}(z)$ satisfy
$
\max\{\frac{y'}{y},\frac{y}{y'}\} \leq 1 + \frac16 \epsilon^2.
$\,
Hence, $G_{\beta}$ has hyperbolic length distortion $\leq (1 + \frac16 \epsilon^2)(1+\frac35 \epsilon^2)$. Multiplying the three bounds we get (iii)
\end{proof}

\textbf{Corollary~\ref{cor:cor1}} is an immediate consequence of \textbf{Theorem~\ref{thm:mainrefined}}: A first mapping with dilatation $\leq 1 + 2\epsilon_3^2$ sends $Y_{l_1,\epsilon_2,\epsilon_3}$ onto $Y_{l_1,\epsilon_2}^{\epsilon_3^*}=Y_{l_1,\epsilon_2,0}^{l_1,\epsilon_2,\epsilon_3^*}$, where  $\gamma_3$ is sent to the horocycle $h_{\epsilon_3^*}$ of length $\epsilon_3^* = \frac{2}{\pi} \epsilon_3 <1$ (\textit{Figure 1}). This horocycle lies in the reduced collar $\hat{\mathcal{C}}_3$ of $Y_{l_1,\epsilon_2,0}$. A second mapping with dilatation $\leq 1 + 2\epsilon_2^2$ sends $Y_{l_1,\epsilon_2,0}$ onto $Y_{l_1,0,0}^{l_1,\epsilon_2^*,0}$. As $\hat{\mathcal{C}}_3$ is mapped isometrically $Y_{l_1,\epsilon_2}^{\epsilon_3^*}$ is sent onto $Y_{l_1,0,0}^{l_1,\epsilon_2^*,\epsilon_3^*}$.

\bigskip

We state two further corollaries. Combining two mappings as in \textbf{Theorem}~\ref{thm:mainrefined} we get
\begin{cor}
Let $0 \leq l_1,l_2$, $0 < \epsilon, \bar{\epsilon} \leq \frac{1}{2}$. Then there exists a boundary coherent homeomorphism between the reduced Y-pieces 
\[
   \phi : \hat{Y}_{l_1,l_2,\epsilon} \rightarrow \hat{Y}_{l_1,l_2,\bar{\epsilon}}
\]
with quasiconformal dilatation $q_{\phi} \leq (1+2 \epsilon^2)(1+2 \bar{\epsilon}^2)$ and length distortion $q'_{\phi}\leq (1+\frac52 \epsilon^2)(1+\frac52 \bar{\epsilon}^2)$.
\label{cor:quasiisom2}
\end{cor}

As the above mappings of the reduced Y-pieces are boundary coherent we may extend them straightforwardly to the adjacent reduced collars in various ways. As an example we have the following which we state without proof.
\begin{cor}
Let $0 \leq l_1,l_2 $, $0 < \bar{\epsilon} \leq \epsilon  \leq \frac{1}{2}$, and set
\[
\delta =\bar{\epsilon}\sec\left\{ \arccos\left(\frac{\bar{\epsilon}}{1+\frac14 \bar{\epsilon}^2}\right) -\frac{\bar{\epsilon}}{\epsilon}\arccos\left(\frac{\epsilon}{1+\frac14 \epsilon^2}\right) \right\}.
\]
Then there exists a boundary coherent quasiconformal homeomorphism
\[
   \phi : Y_{l_1,l_2,\epsilon} \rightarrow Y_{l_1,l_2,\bar{\epsilon}}^{l_1,l_2,\delta}
\]
with dilatation $q_{\phi} \leq (1+2 \epsilon^2)(1+2 \bar{\epsilon}^2)$ that, moreover, acts isometrically on the reduced collars $\hat{\mathcal{C}}_1$, $\hat{\mathcal{C}}_2$ and is conformal on the reduced collar $\hat{\mathcal{C}}_3$ of $Y_{l_1,l_2,\epsilon}$. 
The constant $\delta$ has the estimates
\[
\frac{2}{\pi}\epsilon \cdot \boldsymbol{s}(\tfrac{\bar{\epsilon}}{\epsilon}) - \frac{\epsilon^4}{12} \leq \delta \leq \frac{2}{\pi}\epsilon \cdot \boldsymbol{s}(\tfrac{\bar{\epsilon}}{\epsilon}),
\]
where $\boldsymbol{s}(t) =\frac{\frac{\pi}{2}t}{\sin(\frac{\pi}{2}t)}$, $t \in [0,1]$.
\label{cor:quasiYYbar}
\end{cor}

Finally, we remark that for Y-pieces with large $\ell(\gamma_1)$, $\ell(\gamma_2)$, heuristic arguments indicate that the optimal constant for \textbf{Theorem~\ref{thm:main}} should satisfy $q_{\phi} \geq 1 + K \epsilon^2$, where $K$ is independent of $\epsilon$, $\ell(\gamma_1)$, $\ell(\gamma_2)$. However, we do not have a rigorous proof at this moment and it is also not clear whether it is true in general.


\section*{Acknowledgment}
While working on this article the third author was supported by the Alexander von Humboldt foundation.

\vspace{1cm}

\noindent Peter Buser \\
\noindent Department of Mathematics, Ecole Polytechnique F\'ed\'erale de Lausanne\\
\noindent Station 8, 1015 Lausanne, Switzerland\\
\noindent e-mail: \textit{peter.buser@epfl.ch}\\
\\
\\
\noindent Eran Makover\\
\noindent Department of Mathematics, Central Connecticut State University\\
\noindent 1615 Stanley Street, New Britain, CT 06050, USA\\
\noindent e-mail: \textit{makovere@ccsu.edu}\\
\\
\\
\noindent Bjoern Muetzel and Robert Silhol\\	
\noindent Department of Mathematics, Universit\'e Montpellier 2 \\
\noindent place Eug\`ene Bataillon, 34095 Montpellier cedex 5, France \\
\noindent e-mail: \textit{bjorn.mutzel@gmail.com} and \textit{robert.silhol@math.univ-montp2.fr}\\

\end{document}